\documentclass{amsart}
\usepackage{amsmath,amsthm,amssymb,amsfonts}
\usepackage[colorlinks=true]{hyperref}

\hypersetup{urlcolor=blue, citecolor=blue}

\def\arXiv#1{   {\href{http://arxiv.org/abs/#1}
   {{\mdseries\ttfamily arXiv:#1}}}}

\let\MR\mr

\def\doi#1{   {\href{http://dx.doi.org/#1}
   {{\mdseries\ttfamily DOI}}}}

\newcommand{\al}{\alpha}    \newcommand{\be}{\beta}
\newcommand{\de}{\delta}    
  \newcommand{\ep}{\varepsilon}

\newcommand{\ga}{\gamma}    
\newcommand{\R}{\mathbb{R}}\newcommand{\Z}{\mathbb{Z}}

\newcommand{\pt}{\partial_t}\newcommand{\pa}{\partial}
\newcommand{\les}{{\lesssim}}
\newcommand{\beeq}{\begin{equation}}\newcommand{\eneq}{\end{equation}}

\newcommand{\Sp}{{\mathbb S}}

\newenvironment{prf}{\noindent {\bf Proof.} }{\endprf\par}
\def \endprf{\hfill  {\vrule height6pt width6pt depth0pt}\medskip}

\numberwithin{equation}{section}

 \newcommand{\eps}{\varepsilon}

\def\<{\langle}             \def\>{\rangle}
\def\({\left(}                 \def\){\right)}

\newtheorem{thm}{Theorem}[section]
\newtheorem{prop}[thm]{Proposition}
\newtheorem{coro}{Corollary}
\newtheorem{lem}[thm]{Lemma}

\theoremstyle{remark}
\newtheorem{rem}{Remark}

\theoremstyle{definition}

\title
{The Glassey conjecture for nontrapping obstacles}

\author{Chengbo Wang}
\address{Department of Mathematics\\
                Zhejiang University\\
                Hangzhou 310027, China}
\email{wangcbo@zju.edu.cn}
\urladdr{http://www.math.zju.edu.cn/wang}
\thanks{
The author was supported by Zhejiang Provincial Natural Science Foundation of China LR12A01002, the Fundamental Research Funds for the Central Universities, NSFC 11301478, 11271322 and J1210038.
}
\dedicatory{} \commby{}

\begin{document}

\begin{abstract}
We verify the $3$-dimensional Glassey conjecture for exterior domain $(M,g)$, where the metric $g$ is asymptotically Euclidean, provided that certain local energy assumption is satisfied. The radial Glassey conjecture exterior to a ball is also verified for dimension three or higher. The local energy assumption is satisfied for many important cases, including exterior domain with nontrapping obstacles and flat metric, exterior domain with star-shaped obstacle and small asymptotically Euclidean metric, as well as the nontrapping asymptotically Euclidean manifolds $(\R^n,g)$.
\end{abstract}



\maketitle 

\section{Introduction}
\label{sec-Intro}

The purpose of this paper is to show how integrated local energy estimates for certain
linear wave equations involving asymptotically Euclidean perturbations of the standard Laplacian lead to optimal existence theorems for the corresponding small amplitude
nonlinear wave equations with power nonlinearities in the derivatives. The problem is an analog of the Glassey conjecture in the exterior domain, see Hidano-Wang-Yokoyama \cite{HWY11} and references therein. In particular, for spatial dimension three, we prove global existence of small amplitude solutions for any power greater than a critical power, as well as the almost global existence for the critical power. 
The critical power is the same as that on the Minkowski space. On the other hand, for dimension four and higher, the current technology could only apply for the radial case, and we obtain existence results with certain lower bound of the lifespan, which is sharp in general. The non-radial case is still open, even for the Minkowski space, when the spatial dimension is four or higher.

Let us start by describing the asymptotically Euclidean manifolds $(M,g)$, 
where $M=\R^n\backslash \mathcal{K}$ with smooth and compact obstacle $\mathcal{K}$ and $n\ge 3$. Without loss of generality, when $\mathcal{K}$ is nonempty, we assume the origin lies in the interior of $\mathcal{K}$ 
and $\mathcal{K}\subset B_{1}=\{x\in \R^n: |x|<1\}$. By asymptotically Euclidean, 
we mean that 
\begin{equation}\tag{H1} \label{H1}
g=g_0+g_1(r)+g_2(x), g= g_{ij} (x)  d x^i d x^j =\sum_{i,j=1}^n
g_{ij}(x)  d x^i d x^j
\end{equation}
where $(g_{ij})$ is uniformly elliptic, $(g_{0,ij})=Diag(1,1,\cdots,1)$ is the standard Euclidean metric,  the first perturbation $g_1$ is {\em radial}, and
\beeq\tag{H1.1} \label{H1.1}
\sum_{i j k} \sum_{l\ge 0} 2^{l(i+|\al|-1)} \|\nabla^\al g_{i, jk} \|_{L^\infty_{x}(A_l)} \les 1, \forall \al\ .\eneq
Here, $A_0=\{|x|\le 1\}$, $A_l=\{2^{l-1}\le |x|\le 2^{l}\}$ for $l\ge 1$, and we say $g_1$ is radial, if, when writing out the metric $g$, with $g_2=0$, in polar coordinates $x=r\omega$ with $r=|x|$ and $\omega\in\mathbb{S}^{n-1}$, we have
$$g=g_0+g_1=\tilde g_{11}(r)dr^2+\tilde g_{22}(r) r^2 d\omega^2\ .$$ 
In this form, the assumption \eqref{H1.1} for $g_1$ is equivalent to the following requirement
\begin{equation}\tag{H1.2} \label{H1.2}
\sum_{l\ge 0} 2^{|\al| l}\|\nabla^\al ( \tilde g_{11}-1, \tilde  g_{22}-1)\|_{ L^\infty_{x}(A_l)}
\les 1 
, \forall \al\ .\
\end{equation}
When $g=g_0+\de (g_1+g_2)$ with sufficient small parameter $\de$, we call it a \textit{small} perturbation.
Notice that this sort of assumption and its role in local energy estimates seems to have started with Tataru 
\cite{T08} for Schr\"odinger equations and Metcalfe-Tataru \cite{MeTa07} for wave equations. See also
Tataru \cite{Ta13}, Metcalfe-Tataru-Tohaneanu \cite{Ta12}
for similar assumptions regarding the interaction with rotations.

We shall consider Dirichlet-wave equations on 
$(M,g)$,
\begin{equation}
\label{wave}
\left\{\begin{array}{l} \Box_g u \equiv (\pt^2-\Delta_g)u=F,\ x\in M, t>0\\
u(t,x)=0, x\in \pa M, t>0\\
u(0,x)=\phi(x), \pt u(0,x)=\psi(x)\ ,
\end{array}\right.
\end{equation}
where $\Delta_g$ is the Laplace-Beltrami operator associated with $g$.

Now we can state the local energy assumption that we shall make\\
{\bf Hypothesis 2}. For any $R>1$, we have
\beeq\label{LE}\tag{H2}
\|(\pa u, u)\|_{ L^2_t L^2_x (B_R)}\le C( \|\phi\|_{H^1}+\|\psi\|_{L^2}+\|F\|_{L^2_t L^2_x})\ ,
\eneq 
for any solutions to \eqref{wave} with data $(\phi,\psi)$ and the forcing term $F(t,x)$ vanishes for $|x|>R$.
Here $\pa=(\pt, \nabla)$ is the space-time gradient, and the constant $C$ may depend on $R$.

Let us review some important cases where the assumption \eqref{LE} is valid. First of all, when $g_1=g_2=0$, it is true for any nontrapping obstacle $\mathcal{K}$. In which case, we have $$\|(\pa u(t), u(t))\|_{L^2_x (B_R)}\le \alpha(t) ( \|\phi\|_{H^1}+\|\psi\|_{L^2})$$ with $\al(t)\les \<t\>^{-(n-1)}\in L^1_t\cap L^2_t$, 
for any homogeneous solutions to \eqref{wave} with data $(\phi,\psi)$ supported in $B_R$.
See Melrose \cite{Mel}, Ralston \cite{Ral} and references therein. For the case where $g$ is a compact perturbation of $g_0$, and $M$ is assumed to be nontrapping with respect to the metric, one also has \eqref{LE} for the Dirichlet-wave equation for all $n\ge3$ (Taylor \cite{taylor}, Burq \cite{burq}). For general nontrapping asymptotically Euclidean manifolds without obstacles, it is also known to be true (Bony-H\"afner \cite{BoHa}), at least when
\eqref{H1.1} is replaced by
\beeq\tag{H1.1'} \label{H1.1'} |\nabla^\al g_{1, jk}(x)| \les \<x\>^{
-|\al|-\de},\ 
\sum_{l\ge 0} 2^{l(|\al|+1)} \|\nabla^\al g_{2, jk} \|_{L^\infty_{x}(A_l)} \les 1,
 \eneq for some $\de>0$, where
$\<x\>=\sqrt{1+|x|^2}$. At last, it is known from Metcalfe-Sogge \cite{MeSo06} and Metcalfe-Tataru \cite{MT09} that we still have \eqref{LE}, if $g$ is a small asymptotically Euclidean  metric perturbation, and the obstacle is star-shaped (that is,
$\mathcal{K}=\{r\omega: 0\le  r\le \ga(\omega)<1, \omega\in\Sp^{n-1}\}$, for some smooth positive function $\ga$).

Having described the main assumptions about the linear problem, let us now
turn to the nonlinear equations.
Let $n\ge 3$, $p>1$, we consider the following nonlinear wave equations, 
\begin{equation}
\label{a-pde}
\left\{\begin{array}{l}\Box_g u =a(u)|\pt u|^p+ \sum_{j=1}^n a_j(u) |\pa_j u|^p\equiv F_p(u,\pt u) ,\ x\in M\\
u(t,x)=0, x\in \pa M, t>0\\
u(0,x)=\phi (x), \pt u(0,x)=\psi(x)\ ,
\end{array}\right.
\end{equation}
for given smooth functions $a$ and $a_j$,
as well as the radial problems (with $g_2=0$, $\mathcal{K}=\overline{ B_1}$)
\begin{equation}
\label{a-pde2}
\left\{\begin{array}{l}\Box_g u = a |\pt u|^p+b|\nabla u|^p\equiv G_p(u, \pt u)\ , x\in M\\
u(t,x)=0, |x|=1, t>0\\
u(0,x)=\phi (x), \pt u(0,x)=\psi(x)\ ,
\end{array}\right.
\end{equation} for given constants $a$, $b$.
When $\mathcal{K}$ is empty, it is understood as a Cauchy problem in \eqref{a-pde}.

For such problems posed on the Minkowski space, it is conjectured that  the critical power $p$ for the problem,
to admit global solutions with small, smooth initial data with compact support is $$p_c=1+\frac{2}{n-1}$$ 
in Glassey \cite{Glassey} (see also Schaeffer \cite{Sch86}, Rammaha \cite{Ram87}). The conjecture was verified in dimension $n=2, 3$ for general data (Hidano-Tsutaya \cite{HiTs95} and Tzvetkov \cite{Tz98} independently, as well as the radial case in Sideris \cite{Si83} for $n=3$). For the radial data, the existence results with sharp lower bound on the lifespan for any $p\in (1, 1+2/(n-2))$ was recently proved in Hidano-Wang-Yokoyama \cite{HWY11} (see also Fang-Wang \cite{FaWa10} for the critical case $n=2$ and $p=3$), which particularly verified the Glassey conjecture in the radial case.
On the other hand, for any spatial dimension, the blow up results (together with an explicit upper bound of the lifespan) for \eqref{a-pde},
with $F_p(u,\pt u)=|\pt u|^p$ and $p\le p_c$, were obtained in Zhou \cite{Zh01}, Zhou-Han \cite{ZhHa10} when $g$ is a compact metric perturbation.
Recently, in \cite{Wa13}, the author extended the existence results in 
\cite{HiTs95, Tz98, HWY11} to the setting with small space-time dependent asymptotically flat perturbation of the metric on $\R^n$ with $n\ge 3$, as well as the three dimensional nontrapping asymptotically Euclidean manifolds.

We can now state our main results. The first result is about the problem \eqref{a-pde} with general data,
which verify the $3$-dimensional Glassey conjecture in exterior domains, with asymptotically Euclidean metric perturbation, under the local energy assumption.
\begin{thm}
\label{thm-3}
Let $n=3$, $\mathcal{K}$ be empty or smooth and compact obstacles, and $p>2$. Consider the problem \eqref{a-pde} on $(M , g)$ satisfying \eqref{H1} and \eqref{LE}. 
There exists a small positive constant $\eps_0$, such that the problem \eqref{a-pde} has a unique global solution satisfying $u\in C([0,\infty); H_D^3(M))\cap C^1([0,\infty); H^2(M))$, whenever the initial data
satisfy the compatibility conditions of order $3$, and
\begin{equation}
\label{a-smallness}
\sum_{|\al|\le 2}\|(\nabla,\Omega)^\al (\nabla \phi,\psi)\|_{L^2(M)}=\ep \le  \ep_0, \ \|\phi\|_{L^2(M)}<\infty\ .
\end{equation} 
Moreover, when $p=2$, there exists some $c>0$, so that we have unique solution satisfying $u\in C([0, T_\ep]; H_D^3(M))\cap C^1([0, T_\ep]; H^2(M))$, with $T_\ep=\exp(c /\ep)$.
\end{thm}
The almost global existence result in the case $p=2$ corresponds to the semilinear version of the John-Klainerman theorem \cite{JK84} in $\R^3$ (see Wang-Yu \cite{WY12} and references therein for recent related work for asymptotically Euclidean manifolds), as well as the seminal work of Keel-Smith-Sogge \cite{KSS02} for nontrapping obstacles. Notice that we have considerably improved the required regularity.

Here, by the  compatibility conditions of order $3$, we mean that
\beeq\label{comp1}\phi(x)=0, \psi(x)=0, \Delta_g \phi+F_p(\phi,\psi)=0\eneq
for any $x\in \pa M$. In general, we see from the equation \eqref{a-pde} that, formally, there exist $\Phi_k$ such that
$$\pt^k u(0,x)=\Phi_k(J_k \phi, J_{k-1} \psi)$$
for $x\in M$, where $J_k f=\nabla^{\le k} f\equiv (\nabla^\al f)_{|\al|\le k}$.  Then the  compatibility conditions of order $k+1$ is precisely
$\Phi_j(J_j \phi, J_{j-1} \psi)(x)=0$ for any $x\in\pa M$ and $0\le j\le k$. 
Similarly, for the equation \eqref{wave}, formally, there exist $\tilde \Phi_k$ such that
$$\pt^k u(0,x)=\tilde\Phi_k(J_k \phi, J_{k-1} \psi, J_{k-2}F)$$
for $x\in M$, where $J_k F(x)=\pa^{\le k} F(0,x)$.  Then the  compatibility conditions of order $k+1$ is precisely
$\tilde\Phi_j(J_j \phi, J_{j-1} \psi, J_{j-2}F)(x)=0$ for any $x\in\pa M$ and $0\le j\le k$.

In particular, as special cases, we have the following corollaries, for which, as we have recalled, it is known that we have \eqref{H1} and \eqref{LE}. See \cite{BoHa}, \cite{Mel, Ral}  and 
Lemma \ref{thm-LE-g1-620}
 for the corresponding local energy estimates.
\begin{coro}
Let $M=\R^3$ and $g$ be a nontrapping asymptotically Euclidean perturbation of the flat metric (\eqref{H1} with \eqref{H1.1'}),  then the $3$-dimensional Glassey conjecture is true.
\end{coro}
This recover Theorem 1.1 in \cite{Wa13} for the case of asymptotically Euclidean manifolds. Notice that we have also slightly relaxed the metric assumption.

\begin{coro}
Let $n=3$, $g=g_0$ and $\mathcal{K}$ be empty or a nontrapping obstacle, then the Glassey conjecture is true.
\end{coro}

\begin{coro}
Let $g$ be a small, asymptotically Euclidean perturbation of the flat metric, and $\mathcal{K}$ be a star-shaped obstacle, then the $3$-dimensional Glassey conjecture is true.
\end{coro}

Turning to the problem \eqref{a-pde2} with radial data, we have long time existence of the radial solutions, in spirit of \cite{HWY11}, where the lower bound of the lifespan is sharp in general (\cite{Zh01, ZhHa10}). 
\begin{thm}\label{thm-sub}
Let $n\ge 3$, $p>p_c=1+2/(n-1)$, $\mathcal{K}=\overline{B_1}$, $g_2=0$, $(M , g)$ satisfying \eqref{H1} and \eqref{LE}.
Consider the problem \eqref{a-pde2} with radial data,
there exists a small positive constant $\eps_0$, such that the problem has a unique global  radial solution satisfying $u\in C([0,\infty); H_D^2(M))\cap C^1([0,\infty); H^1(M))$, whenever the initial data
satisfy the compatibility conditions of order $2$, and
\begin{equation}
\label{b-smallness}
\sum_{|\al|\le 1}\|\nabla^\al (\nabla \phi,\psi)\|_{L^2(M)}=\ep \le  \ep_0, \ \|\phi\|_{L^2(M)}<\infty\ .
\end{equation} 
Moreover, when $p\le p_c$, there exist some $c>0$, so that we have unique radial solutions satisfying $u\in C([0, T_\ep]; H_D^2(M))\cap C^1([0, T_\ep]; H^1(M))$, with $T_\ep=\exp(c \ep^{1-p})$ for $p=p_c$ and $T_\ep=c\ep^{2(p-1)/[(n-1)(p-1)-2]}$ for $1<p<p_c$.
\end{thm}
\begin{rem}
The smallness assumption \eqref{b-smallness} on the initial data could be weakened to 
be of ``multiplicative form", as in \cite{HWY11}. 
\end{rem}

As before, it is clear that Theorem \ref{thm-sub} applies for the flat or small asymptotically Euclidean metric, in the domain exterior to a ball.
\begin{coro}
Let $g$ be a small, radial, asymptotically Euclidean perturbation of the flat metric, and $\mathcal{K}=\overline{B_1}$, then the radial Glassey conjecture is true, for dimension $n\ge 3$.
\end{coro}

\begin{rem} Comparing the current Theorem \ref{thm-sub} with Theorem 1.1 in \cite{HWY11}, we do not need to assume $p<1+2/(n-2)$, which, in $\R^n$, is partly due to the $H^2$ regularity. The reason, for us to avoid the restriction in the case of exterior domain, is that we have radial Sobolev embedding $H^1\subset L^\infty$ (see Lemma \ref{thm-trace}), which is not valid in $\R^n$.
\end{rem}

As in \cite{HWY11} and \cite{Wa13}, 
one of the main ingredients in the proof is the local energy estimates with variable coefficients, in spirit of \cite{MeSo06, HWY10}. The local energy estimates first appeared in Morawetz \cite{Mo2},
which are also known as the Morawetz estimates. By now there are extensive literatures devoted to this topic and its applications, without being exhaustive we mention \cite{Strauss75, KPV, SmSo, KSS02, burq, KSS04, Sterb, HY05, MeSo05, MeSo06, MeTa07, HMSSZ, SoWa10, Ta12, Ta13, LMSTW}. Based on \eqref{H1} and \eqref{LE},  we could prove the following version of the local energy estimates. See \eqref{eq-LE} for the notations.
\begin{thm}\label{thm-LE}
Let $(M , g)$ satisfying \eqref{H1} and \eqref{LE}, then 
for any solutions to \eqref{wave}
with $(\phi, \psi, F)\in \dot H^1_D\times L^2_x \times ( LE^*+L^1_t L^2_x)$, we have $u\in C([0,\infty); \dot H^1_D(M))$, and \beeq\label{eq-LE1}\|u\|_{LE\cap E}\les \|\phi\|_{\dot H^1_D}+\|\psi\|_{L^2_x}+\|F\|_{LE^*+L^1_t L^2_x}\ .\eneq
\end{thm}

To prove the existence results, we need the following higher order local energy estimates, 
\begin{prop}[Higher order local energy estimates]\label{thm-LE-h}
For $(M , g)$ satisfying \eqref{H1} and \eqref{LE},  there exists $R>4$ such that, we have
\begin{eqnarray}
\|u\|_{LE_k\cap E_k}&\les& \sum_{|\al|\le k}\|(\nabla, \Omega)^\al (\nabla \phi,\psi)\|_{L^2_x}+\|Z^\al F\|_{LE^*+L^1_t L^2_x}\label{eq-LE2}\\
&&
+\sum_{|\ga|\le k-1}\|Z^\ga F(0,x)\|_{L^2_x}+\|\pa^\ga F\|_{(L^\infty_t\cap L^2_t )L^2_x(B_{2R})}\ ,\nonumber
\end{eqnarray}
for any solutions to \eqref{wave} satisfying compatibility condition of order $k+1$. Here and in what follows, $B_R$ means $\{x\in M: |x|<R\}$.
\end{prop}

For the existence results with $p\le p_c$, we will also require a relation between the KSS type estimates \cite{KSS02, HY05, MeSo06} and the local energy estimates.
Basically,  it is known that, the local energy norm, together with the energy norm, could control the KSS-type norm, see, e.g., \cite{KSS02}, \cite{M04}, \cite{MeSo06} and \cite{Wa13} Lemma 3.4.  
Moreover, a dual version also holds, see e.g., \cite{MeTa07}.
\begin{lem}\label{KSS-dual}
For any $\mu\in [0,1/2)$, there are positive constants $C_{\mu}$ and $C$, independent of $T\ge 2$, such that
\beeq\label{eq-KSS}
\|\pa u\|_{l^{-1/2}_2(L^2_T L^2_x)}+\|r^{-1} u\|_{l^{-1/2}_2(L^2_T L^2_x)}
\le C (\ln T)^{1/2} \|u\|_{LE\cap E ([0,T]\times M)}\ ,
\eneq
\beeq\label{eq-KSS2}
\|\pa u\|_{l^{-\mu}_2(L^2_T L^2_x)}+\| r^{-1} u\|_{l^{-\mu}_2(L^2_T L^2_x)}
\le C_{\mu} T^{1/2-\mu}\|u\|_{LE\cap E ([0,T]\times M)}\ .
\eneq
Moreover, we have 
\beeq\label{eq-KSS'}
\|F\|_{LE^*+L^1_T L^2_x ([0,T]\times M)}
\le
 C (\ln T)^{1/2} \|F\|_{l^{1/2}_2(L^2_T L^2_x)}
\ ,
\eneq
\beeq\label{eq-KSS2'}
\|F\|_{LE^*+L^1_T L^2_x ([0,T]\times M)}
\le
C_{\mu} T^{1/2-\mu}
\|F\|_{l^{\mu}_2(L^2_T L^2_x)}\ .
\eneq
Here we use $L^q_T$ to stand for $L^q_t([0,T])$.
\end{lem}

This paper is organized as follows. In the next section, we recall some Sobolev type estimates, in relation with trace theorem and Hardy's inequality. In Section 3, we give the proof of the local energy estimates, Theorem \ref{thm-LE} and Proposition \ref{thm-LE-h}, based on \eqref{H1} and \eqref{LE}, as well as a relation between the local energy estimates and KSS type estimates, Lemma \ref{KSS-dual}. In the fourth section, we give the proof of the three dimensional Glassey conjecture, following the approach of \cite{HWY11, Wa13}, adapted in the setting of exterior domains. In the last section, we prove the radial Glassey conjecture. 

\subsection{Notations} Finally we close this section by listing the notations.

 \noindent $\bullet$ $A\les B$ means that $A\le C B$ where the constant $C$ may change from line to line.

 \noindent $\bullet$
$(x^0,x^1,\cdots ,x^n)=(t,x)\in \R^{1+n}$, and $\partial_i={\partial}/{\partial x^i}$, $0\le i\le n$, with the abbreviations $\partial=(\partial_0,\partial_1,\cdots ,\partial_n)=(\partial_t,\nabla)$.
$\pa^\al=\pa_0^{\al_0}\cdots \pa_n^{\al_n}$ with multi-indices $\al, \be\in \Z_+^{n+1}$. 
 The vector fields to be used will be labeled as
\[
Y=(Y_1, \cdots, Y_{n(n+1)/2})=(\nabla,\Omega), Z=(\pt, Y)
\]
with rotational vector fields
$\Omega_{ij}=x_i \pa_j-x_j\pa_i, 1\le i<j\le n$.
Sometimes, we use $Z^{\le k}$ to denote $(Z^\al)_{|\al|\le k}$.

 \noindent $\bullet$ 
With Dirichlet boundary condition, we define $\dot H^1_D(M)$ as the closure of $f\in C_0^\infty(M)$, with respect to the norm
$$\|f\|_{\dot H^1_D(M)}=\|\nabla f\|_{L^2(M)}\ .$$
When $M=\R^n$, $\dot H^1$ means the closure of $C_0^\infty$ with respect to the $\dot H^1$ norm.

 \noindent $\bullet$
The space $l^s_q(A)$ ($1\le q\le \infty$) means
 \[
 \|u\|_{l^s_q(A)} = \|(\Phi_j(x)u(t, x))\|_{l^s_q(A)}=\|\left(\|2^{js}\Phi_j(x)u(t,x)\|_A\right)\|_{l^q_{j\ge 0}},
\]
for a partition of unity subordinate to the (inhomogeneous) dyadic (spatial) annuli,
$\sum_{j\ge 0}\Phi_j(x)=1$. Typical choice could be a radial, nonnegative $\Phi_0(x)\in C_0^\infty$ with value $1$ for $|x|\le 1$, and $0$ for $|x|\ge 2$, and $\Phi_j(x)=\Phi(2^{-j}x)-\Phi(2^{1-j}x)$ for $j\ge 1$.

 \noindent $\bullet$ 
$\|\cdot\|_{E_m}$  is the energy norm of order $m\ge 0$,
\beeq\label{eq-E}\|u\|_{E}=\|u\|_{E_0}=\|\pa u\|_{L^\infty_t L^2_x (\R_+\times M)}\ ,
\|u\|_{E_{m}}=\sum_{|\al|\le m}\| Z^\al u\|_{E}\ .\eneq
Also, we use $\|\cdot\|_{LE}$ to denote the  local energy norm
\beeq\label{eq-LE}
\|u\|_{LE}=\|\pa u\|_{l^{-1/2}_\infty L^2_t L^2_x(\R_+\times M)}+\|u/r\|_{l^{-1/2}_\infty L^2_t L^2_x (\R_+\times M)},
\eneq
On the basis of the local energy norm, we can similarly define
$\|u\|_{LE_m}$,  and the dual norm $LE^*=l^{1/2}_1 L^2_t L^2_x(\R_+\times M)$.

 \noindent $\bullet$ 
$\|u\|_{X+Y}=\inf_{u=u_1+u_2}(\|u_1\|_{ X}+\|u_2\|_{Y})$

 \noindent $\bullet$ 
For fixed $R>1$, let $\be(x)=\Phi_0(x/R)\in C_0^\infty$ such that $\be=1$ for $|x|\le R$ and vanishes for $|x|\ge 2R$. Based on $\be$, we set $\be_1(x)=\be(x/2)$, $\be_2(x)=\be(2x)$, \beeq\label{eq-gg}\tilde g=\be(4x)g_0+(1-\be(4x))g=g_0+(1-\be(4x))(g_1+g_2)\ ,\eneq which agrees with $g$ for $|x|\ge R/2$ and $g_0$ for $|x|\le R/4$.
Notice that for these functions, we have
$(1-\be)(1-\be_1)=1-\be_1$,
$\Box_g (1-\be) u=\Box_{\tilde g} (1-\be) u$,
$\Box_g (1-\be_2) u=\Box_{\tilde g} (1-\be_2) u$.

\section{Sobolev-type estimates}
\label{sec-Sobolev}
In this section, we recall several Sobolev type estimates in relation with the trace
theorem and Hardy's inequality. At first, we have the following trace theorem (see Lemma 2.2 in \cite{HWY11},  (1.3), (1.7) in \cite{FaWa11} and references therein)
\begin{lem}\label{thm-trace}
Let $n\ge 2$, then
\beeq\label{eq-trace}
\|r^{(n-1)/2} u(r\omega)\|_{L^{2}_\omega}\les \|u\|_{L^2(|x|\ge r)}+\|\nabla u\|_{L^2(|x|\ge r)}\ .
\eneq
\end{lem}

We will also need the following variant of the Sobolev embeddings.
\begin{lem}\label{thm-Sobo}
Let $n\ge 2$.
For any $m\in \R$ and $k\ge n/2-n/q$ with $q\in [2,\infty)$, we have
\beeq\label{eq-Sobo}
\|\<r\>^{(n-1)(1/2-1/q)+m} u\|_{L^q(M)}\les \sum_{|a|\le k} \|\<r\>^{m} Y^a u\|_{L^2(M)}\ .
\eneq
Moreover, we have
\beeq\label{eq-Sobo2}
\|\<r\>^{(n-1)/2+m} u\|_{L^\infty(M)}\les \sum_{|a|\le [(n+2)/2]} \|\<r\>^{m}  Y^a u\|_{L^2(M)}\ ,
\eneq
where $[a]$ stands for the integer part of $a$.
\end{lem}
When $M=\R^n$, it is precisely Lemma 2.2 in \cite{Wa13} (see also Lemma 3.1 in  \cite{LMSTW}). For the exterior domain, the estimates follow from a simple cutoff argument and the classical Sobolev embedding.

When dealing with \eqref{a-pde}, we need to have a local control of $u$, from $\nabla u$, which is achieved by the Hardy inequality.
\begin{lem}[Hardy's inequality]\label{thm-Hardy} 
Let $n\ge 3$ and $M=\R^n\backslash \mathcal{K}$ with smooth and compact $\mathcal K$.
Then for any $u\in \dot H^1_D(M)$, we have
\beeq\label{eq-Hardy}\|u/r\|_{L^2(M)}\les \|\nabla u\|_{L^2(M)}\ .\eneq
\end{lem}
\begin{proof}It is classical, see e.g., Colin \cite{Colin01},
Chabrowski-Willem \cite{CW06}. For reader's convenience, we give an explicit proof in the case of star-shaped obstacle here.
By density, it suffices to prove \eqref{eq-Hardy} for $u\in C_0^\infty(M)$. For this case, we have
\begin{eqnarray*}
\int_{\ga(\omega)}^\infty |u/r|^2 r^{n-1} dr & = &
\frac{1}{n-2}\int_{\ga(\omega)}^\infty u^2 \pa_r r^{n-2} dr
 \\
 & = & \left.\frac{1}{n-2}u^2  r^{n-2}\right|_{r={\ga(\omega)}}^\infty-\frac{2}{n-2}\int_{\ga(\omega)}^\infty  r^{n-2} u \pa_r u dr\\
 & \le & \frac{2}{n-2}\left(\int_{\ga(\omega)}^\infty |u/r|^2 r^{n-1} dr \right)^{1/2}\left(\int_{\ga(\omega)}^\infty |\pa_r u|^2 r^{n-1} dr\right)^{1/2} \ ,
\end{eqnarray*}
which, after integrating with respect to $\omega$,  yields \eqref{eq-Hardy}.
\end{proof}

As a direct consequence, we have
\begin{prop}\label{thm-1212} Let $n=3$ and $u\in \dot H^1_D(M)\cap \dot H^2(M)$,  we have 
\beeq\label{eq-1212}
\|u\|_{L^\infty(M)}\les \sum_{|\al|\le 1}\|\nabla \nabla^\al u\|_{L^2(M)}
\eneq
\end{prop}
\begin{proof}
Since $\mathcal{K}\subset B_1$ and $R>1$,
we can view $(1-\be) u$ as a function in $\R^n$.
By Sobolev embedding $H^2(M\cap B_{2R})\subset L^\infty(M\cap B_{2R})$, and $\dot H^1\cap \dot H^2\subset L^\infty(\R^n)$, we have
\begin{eqnarray*}
\|u\|_{L^\infty(M)} & \le &
\|\be u\|_{L^\infty (M\cap B_{2R})} +\|(1-\be) u\|_{L^\infty(\R^n)} 
 \\
 & \les & 
 \|\be u\|_{H^2(M\cap B_{2R})} +\|(1-\be) u\|_{\dot H^1\cap \dot H^2(\R^n)} 
 \\
 & \les & \|u\|_{L^2(M\cap B_{2R})}+
 \sum_{|\al|\le 1}\|\nabla \nabla^\al u\|_{L^2(M)}
\\
 & \les & 
 \sum_{|\al|\le 1}\|\nabla \nabla^\al u\|_{L^2(M)}
  \ ,
\end{eqnarray*}
where we used Hardy's inequality in the last step. 
\end{proof}

\section{Local energy estimates}\label{sec-KSS}
 In this section, we give the proof of the local energy estimates, Theorem \ref{thm-LE} and Proposition \ref{thm-LE-h}, based on \eqref{H1} and \eqref{LE}. In addition, we prove Lemma \ref{KSS-dual}.

\subsection{Local energy estimates with variable coefficients}
To begin, let us recall local energy estimates with variable coefficients, which are essentially obtained in
\cite{MeSo06}, \cite{MT09} (see also \cite{MeTa07, HWY10, HWY11, WY12} and \cite{Wa13} Lemma 3.1). 
\begin{lem}\label{thm-LE-g1-620}
Let $n\ge 3$ and $M=\R^n$. Consider the linear problem $\Box_g u=F$ with $g(x)=g_0+\delta h(x)$ satisfying $$\sum_{j k} \sum_{l\ge 0} 2^{ l |\al|} \|\partial^\al_{x} h_{jk} \|_{L^\infty_{x}(A_l)} \les 1, \forall \al\ .$$ Then there exists a constant $\de_0$, such that for any $0\le \de\le \de_0$, we have the following local energy estimates,
\beeq\label{eq-LE-g1-620}
\|u\|_{LE\cap E}\les \|\pa u(0)\|_{L^2_x}+\|F\|_{LE^*+L^1_t L^2_x}\ .
\eneq
In addition, the same results apply for
solutions to \eqref{wave}, when $M=\R^n\backslash \mathcal{K}$ with star-shaped $\mathcal{K}$. 
\end{lem}
 Notice that by the assumption, we have $$\Box_g=\Box-r_0^{ij}(x)\pa_i\pa_j+r_1^j(x) \pa_j\ ,$$ where
$\|\pa^\al r_0(x)\|_{l^{|\al|}_1 L^\infty_x}\les \de$, $\|\pa^\al r_1\|_{l^{|\al|+1}_1 L^\infty_x}\les \de$, for all $\al$.
With this observation, the case $M=\R^n$ follows from \cite{MeTa07}. In the case of star-shaped obstacle, we need only to observe further that the boundary term will be of favorable sign and can be disregarded, see \cite{MeSo06, MT09}. We omit the details here.

\subsection{Local energy estimates in exterior domain}

With Lemma \ref{thm-LE-g1-620} at hand, we could give the proof of Theorem \ref{thm-LE}. 
First of all, by Duhamel's principle, it suffices to prove \beeq\label{eq-LE1*}\|u\|_{LE\cap E}\les  \|\phi\|_{\dot H^1_D}+\|\psi\|_{L^2_x}+\|F\|_{LE^*}\ \eneq for solutions to \eqref{wave}. We divide the proof into three steps: 
controlling the local part, the local energy, and the energy.

\subsubsection{Controlling the local part}
At first, we notice that it is possible to choose $R_0\ge 4$ large enough such that, 
$\tilde g$, as defined in \eqref{eq-gg}, satisfies the condition in Lemma \ref{thm-LE-g1-620} for any $R\ge R_0$. As a consequence, we have
\beeq\label{eq-LE-small}\|u\|_{LE\cap E}\les \|\pa u(0)\|_{L^2_x}+\|\Box_{\tilde g} u\|_{LE^*+L^1_t L^2_x}\ .\eneq

 Now, we define $u_1$ as the solution of the Dirichlet-wave equation with data $(\be_1 \phi, \be_1 \psi)$ and forcing term $\be_1 F$, and $u_2=u-u_1$.

For $u_1$, we have trivially 
\beeq\label{eq-LE*1}\|(\pa u_1, u_1)\|_{ L^2_t L^2_x (B_R)}\les \|\be_1\phi\|_{H^1}+\|\be_1\psi\|_{L^2}+\|\be_1F\|_{L^2_t L^2_x}
\les  \|\phi\|_{\dot H^1_D}+\|\psi\|_{L^2_x}+\|F\|_{LE^*}
\ ,
\eneq
by \eqref{LE} and the Hardy inequality \eqref{eq-Hardy}.

To estimate $u_2$, 
we introduce $u_0$ as the solution of the Cauchy problem in $\R^n$
$$\Box_{\tilde g} u_0=(1-\be_1)F, 
u_0(0,x)=(1-\be_1)\phi,
\pt u_0(0,x)=(1-\be_1)\psi.
$$
For $u_0$, we know from \eqref{eq-LE-small} that,
\beeq\label{eq-LE*2}\|u_0\|_{LE}\les \|\phi\|_{\dot H^1_D}+\|\psi\|_{L^2_x}+\|F\|_{LE^*}\ .\eneq
 Now, similar to \cite{SmSo}, let $w=u_2-(1-\be) u_0$, noticing that
 $$
 \Box_g [(1-\be)u_0]= \Box_{\tilde g} [(1-\be)u_0]=(1-\be)\Box_{\tilde g} u_0+[\Delta_{\tilde g}, \be]u_0 
=  (1-\be_1)F+[\Delta_{\tilde g}, \be]u_0 \ ,
 $$ it is easy to see that 
$$\Box_g w =[\beta,\Delta_{\tilde g}] u_0, w|_{\pa M}=0, w(0,x)=0, \pt w(0,x)=0$$
due to the support properties of $\mathcal{K}$, $\be$. Noticing that $[\beta,\Delta_{\tilde g}] u_0$ is supported in $|x|\le 2R$, we could apply \eqref{LE} to obtain
\beeq\label{eq-LE*3}\|(\pa w, w)\|_{ L^2_t L^2_x (B_R)}\les \|
[\beta,\Delta_{\tilde g}] u_0
\|_{L^2_t L^2_x}\les \|u_0\|_{LE} \ .\eneq
Recalling $u=u_1+u_2=u_1+w+(1-\be)u_0$, \eqref{eq-LE*1}-\eqref{eq-LE*3}, we arrived at
\beeq\label{eq-LE-local}\|(\pa u, u)\|_{ L^2_t L^2_x (B_R)}
\les \|\phi\|_{\dot H^1_D}+\|\psi\|_{L^2_x}+\|F\|_{LE^*}\ .\eneq

\subsubsection{Controlling the local energy}
Turning to the full local energy estimates, we divide $u$ into $\be_2 u+(1-\be_2)u$.
For $(1-\be_2)u$, due to the support property, and $\tilde g$ agrees with $g$ for $|x|\ge R/2$, we observe that
$$\Box_{\tilde g} (1-\be_2)u=\Box_g (1-\be_2)u=(1-\be_2)F+[\Delta_g, \be_2] u\ .$$
Viewing $(1-\be_2)u$ as a solution of the Cauchy problem, we 
get from \eqref{eq-LE-small} that
$$\|u\|_{LE}\les \|\be_2 u\|_{LE}+\|(1-\be_2)u\|_{LE}\les \|\pa (1-\be_2) u(0)\|_{L^2_x}+\|(1-\be_2) F\|_{LE^*}
+\|(\pa u, u)\|_{ L^2_t L^2_x (B_R)}\ .
$$
There, applying \eqref{eq-LE-local}, we get
\beeq\label{eq-LE1**}
\|u\|_{LE}\les  \|\phi\|_{\dot H^1_D}+\|\psi\|_{L^2_x}+\|F\|_{LE^*}\ ,\eneq
which is the local energy part of \eqref{eq-LE1*}.

\subsubsection{Controlling the energy}
It remains to control the energy norm in \eqref{eq-LE1*}. For this, we introduce a modified energy norm
$$A(t)=\left(\int_M \frac{u_t^2(t,x)+g^{ij}(x)\pa_i u(t,x) \pa_j u(t,x)}{2} \sqrt{|g|} dx\right)^{1/2}\ ,$$
where $|g|$, $(g^{ij})$ are the determinant and inverse matrix to the matrix $(g_{ij})$. From geometrical point of view, it is a natural definition of the energy. By the uniform elliptic assumption, it is equivalent to the classical energy norm $E$. For $A(t)$, we know from the definition, after integration by parts and noticing that $\pt u|_{\pa M}=0$, that
\beeq\label{eq-E2}\frac{d A(t)^2}{dt}=\int_M u_t F \sqrt{|g|} dx\ .\eneq
After integration in time, we get for any $T$,
$$A^2(T)\le A^2(0)+\int_0^T\int_M |u_t F| \sqrt{|g|} dxdt\les 
 \|\phi\|_{\dot H^1_D}^2+\|\psi\|_{L^2_x}^2
+\|u\|_{LE}\|F\|_{LE^*}$$
Applying \eqref{eq-LE1**}, we know that
$$\|\pa u(T)\|_{L^2_x}^2\les A^2(T)\les  \|\phi\|_{\dot H^1_D}^2+\|\psi\|_{L^2_x}^2+\|F\|_{LE^*}^2$$
and so
$$\|u\|_{LE\cap E}
\les  \|\phi\|_{\dot H^1_D}+\|\psi\|_{L^2_x}+\|F\|_{LE^*}\ ,$$
which is \eqref{eq-LE1*}.
This completes the proof of Theorem \ref{thm-LE}.

\subsection{Higher order estimates}
In this subsection, we give the proof of the higher order local energy estimates, Proposition \ref{thm-LE-h}, based on  Theorem \ref{thm-LE} and Lemma \ref{thm-LE-g1-620}.

As usual, part of the difficulty comes from the fact that the vector fields do not preserve the boundary condition $u|_{\pa M}=0$ in general. Despite of the difficulty, we know that $\pt$ preserves the boundary condition and commutates with the equation. As a consequence, provided the solution to \eqref{wave} satisfies the compatibility condition of order $k+1$, by Theorem \ref{thm-LE}, we have
\begin{eqnarray}
\sum_{0\le j\le k}\|\pt^j u\|_{LE\cap E}&\les&  \sum_{|\al|\le k}\|\nabla^\al (\nabla \phi, \psi)\|_{L^2_x}+\sum_{|\ga|\le k-1}\|\pa^\ga F(0,x)\|_{L^2_x}\label{eq-LE12h1}
\\
&&+\sum_{0\le j\le k} \|\pt^j F\|_{LE^*+L^1_t L^2_x}
\nonumber
\ .\end{eqnarray}
Here,  we have expressed the initial data of $\pt^j u$, through the equation \eqref{wave}, by the combination of $\nabla^\al \phi$, $\nabla^\al \psi$ and $\pa^\al F(0,x)$.

To extend the vector field from $\pt$ to  $Z$, we observe first
$$\|Z^\al u\|_{LE\cap E}\les
\|Z^\al \be_2 u\|_{LE\cap E}+\|Z^\al (1-\be_2 ) u\|_{LE\cap E}\ .
$$
For the second term, $\|Z^\al (1-\be_2 ) u\|_{LE\cap E}$, notice that
$$\Box_{\tilde g} Z^\al (1-\be_2 ) u=
[\Box_{\tilde g}, Z^\al] (1-\be_2 ) u-Z^\al [\Box_{\tilde  g}, \be_2] u+Z^\al (1-\be_2 ) F\ .$$
For
$[\Box_{\tilde g}, Z^\al]$, by \eqref{H1.1}, we know that, for any given $\delta>0$, there exists $R_1\ge R_0$, such that for any $R \ge R_1$, there exists $c_i(x)$ such that
$$|[\Box_{\tilde g}, Z^\al]v|\le c_1(x) \sum_{|\ga|\le |\al|} |Z^\ga \pa v| +c_2 (x)
\sum_{|\ga|\le |\al|} |Z^\ga v|
$$ with $\|c_i(x)\|_{l^{i}_1 L^\infty_x}\le \de$. Here, we used the fact that the first perturbation is radial, which commutates with the rotational vector fields $\Omega$.

Applying  Lemma \ref{thm-LE-g1-620}, together with these information,
\begin{eqnarray*}
\|Z^\al (1-\be_2 ) u\|_{LE\cap E} & \les & 
\sum_{|\ga|\le |\al|}\|(\nabla,\Omega)^\ga (\nabla \phi, \psi)\|_{L^2_x}
+\sum_{|\ga|\le |\al|-1}\|Z^\ga F(0,x)\|_{L^2_x}
 \\
 &  & +
\| [\Box_{\tilde g}, Z^\al] (1-\be_2 ) u\|_{LE^*}+\sum_{|\ga|\le |\al|+1}\|\pa^\ga  u\|_{L^2_t L^2_x (B_R)}
 \\
&&
 +\|Z^\al (1-\be_2 ) F\|_{LE^*+L^1_t L^2_x} \\
 & \les & 
\sum_{|\ga|\le |\al|}\|(\nabla,\Omega)^\ga (\nabla \phi, \psi)\|_{L^2_x}
 +
\sum_{|\ga|\le |\al|-1}\|Z^\ga F(0,x)\|_{L^2_x}
 \\
 &  & 
 + \de \sum_{|\ga|\le |\al|}\| Z^\ga (1-\be_2 ) u\|_{LE}+
\sum_{|\ga|\le |\al|+1}\|\pa^\ga  u\|_{L^2_t L^2_x (B_R)}
\\ 
&&
 +\sum_{|\ga|\le |\al|}\|Z^\ga  F\|_{LE^*+L^1_t L^2_x} \ .
\end{eqnarray*}
Summing over $|\al|\le k$ and setting $\de$ small enough to be absorbed by the left, we conclude that
\begin{eqnarray}
\|u\|_{LE_k\cap E_k} & \le & \|\be_2 u\|_{LE_k\cap E_k}+\| (1-\be_2 ) u\|_{LE_k\cap E_k}\nonumber\\
 & \les & \sum_{|\ga|\le k}\|(\nabla,\Omega)^\ga (\nabla \phi, \psi)\|_{L^2_x}
 +\|Z^\ga  F\|_{LE^*+L^1_t L^2_x} \nonumber
 \\
 &  & +\sum_{|\ga|\le k-1}\|Z^\ga F(0,x)\|_{L^2_x}
+
\sum_{|\ga|\le k}\|\pa^\ga  u\|_{LE\cap E (B_R)}\ .\label{eq-LE-h1}
\end{eqnarray}

To complete the proof of Proposition \ref{thm-LE-h}, it suffices to give the control of the last term in \eqref{eq-LE-h1}.

\subsubsection{Controlling the local part}
Let us prove Proposition \ref{thm-LE-h}, by \eqref{eq-LE12h1}, \eqref{eq-LE-h1},  and induction.

The case $k=0$ follows from Theorem \ref{thm-LE}. Assume it is true for some $k=j\ge 0$, then for $k=j+1$, 
since the problem satisfies the compatibility condition of order $j+2$, we have 
 the compatibility condition of order $j+1$ for $w=\pt u$, and
 $$\Box_g w=\pt F, w|_{\pa M}=0, w(0,x)=\psi, \pt w(0,x)=\Delta_g\phi+F(0,x)\ .$$

At first, we observe that
$$\|\pa^\ga  \pa^2 u\|_{(L^2_t\cap L^\infty_t) L^2_x (B_R)}\les \|\pa^\ga \pa w\|_{(L^2_t\cap L^\infty_t) L^2_x (B_R)}
+
\|\pa^\ga \nabla^2 u\|_{(L^2_t\cap L^\infty_t) L^2_x (B_R)}\ ,$$
with $|\ga|= j$.
For the second term, using elliptic estimate, we get
\begin{eqnarray*}
\|\pa^\ga \nabla^2 u\|_{L^2_x (B_R)} & \les & \|\Delta_g \pa^\ga u\|_{ L^2_x (B_{2R})}+\| \pa^\ga u\|_{ L^2_x (B_{2R})} \\
 & \les & \| \pa^\ga \Delta_g u\|_{ L^2_x (B_{2R})}+\sum_{|\al|\le j+1=k} \|  \pa^\al u\|_{ L^2_x (B_{2R})} \\
 & \les & \| \pa^\ga \pt^2 u\|_{ L^2_x (B_{2R})}+\| \pa^\ga F\|_{ L^2_x (B_{2R})}+
 \sum_{|\al|\le j+1} \|  \pa^\al u\|_{ L^2_x (B_{2R})}\ ,
\end{eqnarray*}
where in the last inequality, we used the equation \eqref{wave}.

In conclusion, we get
\begin{eqnarray}
\|\pa^\ga  \pa^2 u\|_{(L^2_t\cap L^\infty_t) L^2_x (B_R)} & \les & 
\|\pa^\ga \pa w\|_{(L^2_t\cap L^\infty_t) L^2_x (B_R)}
+
\|\pa^\ga \nabla^2 u\|_{(L^2_t\cap L^\infty_t) L^2_x (B_R)} \nonumber\\
 & \les & 
 \| \pa^\ga \pa w\|_{(L^2_t\cap L^\infty_t) L^2_x (B_{2R})}
+\| \pa^\ga F\|_{ (L^2_t\cap L^\infty_t) L^2_x (B_{2R})}\nonumber\\
&&+
 \sum_{|\al|\le j+1} \|  \pa^\al u\|_{(L^2_t\cap L^\infty_t) L^2_x (B_{2R})} \nonumber\\
 & \les & 
 \|w\|_{LE_j\cap E_j}+ \|u\|_{LE_j\cap E_j}
+\| \pa^\ga F\|_{ (L^2_t\cap L^\infty_t) L^2_x (B_{2R})} \label{eq-1214-1}\ ,
\end{eqnarray}
where, in the last inequality, we have used the Hardy inequality, Lemma \ref{thm-Hardy}.

By \eqref{eq-1214-1} and the induction assumption, we have
\begin{eqnarray*}
\sum_{|\ga|\le j+1}\|\pa^\ga  u\|_{LE\cap E (B_R)}
 &\les &
\|u\|_{LE_{j}\cap E_{j}}+ 
\sum_{|\ga|= j}\|\pa^\ga \pa^2  u\|_{(L^2_t\cap L^\infty_t) L^2_x (B_R)}\\
 &\les &
\|u\|_{LE_{j}\cap E_{j}}+\| w\|_{LE_j\cap E_j}+ \sum_{|\ga|= j}\|\pa^\ga F\|_{(L^\infty_t\cap L^2_t )L^2_x(B_{ 2R})}\\
&\les&
\sum_{|\ga|\le j+1}\|(\nabla,\Omega)^\ga (\nabla \phi, \psi)\|_{L^2_x}
 +\|Z^\ga  F\|_{LE^*+L^1_t L^2_x} 
 \\
 &  & +\sum_{|\ga|\le j}\|Z^\ga F(0,x)\|_{L^2_x}
+ \sum_{|\ga|\le j}\|\pa^\ga F\|_{(L^\infty_t\cap L^2_t )L^2_x(B_{ 2R})}\ .
\end{eqnarray*}
Then, by \eqref{eq-LE-h1} with $k=j+1$, $\|u\|_{LE_{j+1}\cap E_{j+1}}$ is controlled by 
\begin{eqnarray*}
 & &\sum_{|\ga|\le j+1}\|(\nabla,\Omega)^\ga (\nabla \phi, \psi)\|_{L^2_x}
 +\|Z^\ga  F\|_{LE^*+L^1_t L^2_x} +\sum_{|\ga|\le j}\|Z^\ga F(0,x)\|_{L^2_x}
 \\
 &  & 
+
\sum_{|\ga|\le j+1}\|\pa^\ga  u\|_{LE\cap E (B_R)}\\
 &\les &\sum_{|\ga|\le j+1}\|(\nabla,\Omega)^\ga (\nabla \phi, \psi)\|_{L^2_x}
 +\|Z^\ga  F\|_{LE^*+L^1_t L^2_x}+\sum_{|\ga|\le j}\|Z^\ga F(0,x)\|_{L^2_x}
 \\
 &  & 
+ \sum_{|\ga|\le j}\|\pa^\ga F\|_{(L^\infty_t\cap L^2_t )L^2_x(B_{ 2R})}\ .
\end{eqnarray*}
This completes the proof of Proposition \ref{thm-LE-h}.

\subsection{A relation between KSS type norm and local energy norm}
In this subsection, for reader's convenience, we give a proof of \eqref{eq-KSS'} and \eqref{eq-KSS2'} in Lemma \ref{KSS-dual}.

As usual, we use a cutoff argument (\cite{KSS02}). Let $F_1=F \chi_{|x|\le T}$ and $F_2=F-F_1$.
\begin{eqnarray*}
\|F\|_{LE^*+L^1_T L^2_x }& \les & 
\|F_1\|_{LE^*}+\|F_2\|_{L^1_T L^2_x}
 \\
 & \les &
\|2^{j/2} F_1(t,x) \Phi_j(x)\|_{l^1_j L^2_T L^2_x}+T^{-1/2}\||x|^{1/2}F_2\|_{L^1_T L^2_x}
 \\
 & \les &
\|2^{j/2} F_1(t,x) \Phi_j(x)\|_{l^2_j L^2_T L^2_x} \|1\|_{l^2_{1\le j\le \ln T}}+\||x|^{1/2} F_2\|_{L^2_T L^2_x}
 \\
 & \les &
(\ln T)^{1/2}\|F\|_{l^{1/2}_2(L^2_T L^2_x)}\ .
\end{eqnarray*}
Similarly,
\begin{eqnarray*}
\|F\|_{LE^*+L^1_T L^2_x }& \les & 
\|2^{j/2} F_1(t,x) \Phi_j(x)\|_{l^1_j L^2_T L^2_x}+T^{-\mu}\||x|^{\mu}F_2\|_{L^1_T L^2_x}
 \\
 & \les &
\|2^{j\mu} F_1(t,x) \Phi_j(x)\|_{l^2_j L^2_T L^2_x} \|2^{j(1/2-\mu)}\|_{l^2_{1\le j\le \ln T}}+T^{1/2-\mu}\||x|^{\mu}F_2\|_{L^2_T L^2_x}
 \\
 & \les &
T^{1/2-\mu}\|F\|_{l^{\mu}_2(L^2_T L^2_x)}\ .
\end{eqnarray*}
This completes the proof.

\section{Glassey conjecture with dimension $3$}\label{sec-Gla-3}
In this section, we will prove Theorem \ref{thm-3}, mainly based on
Lemma \ref{thm-Sobo} and Proposition \ref{thm-LE-h}.

As usual, we shall use iteration to give the proof. We set $u_0\equiv 0$ and recursively define $u_{k+1}$ ($k\ge 0$) be the solution to the linear equation
$$\Box_g u_{k+1}=F_p(u_k, \pt u_k), u_{k+1}(t,x)|_{\pa M}=0, u_{k+1}(0,x)=\phi(x), \pt u_{k+1}(0,x)=\psi(x).
$$ Note that the compatibility condition \eqref{comp1} ensures that, we still have the compatibility condition of order $3$ for $u_{k+1}$, and we can apply Proposition \ref{thm-LE-h}.

By the smallness condition \eqref{a-smallness} on the data, there is a constant $C_1$ so that $\|u_1\|_{LE_2\cap E_2} \le C_1\ep$, and
\begin{eqnarray*}
\|u\|_{LE_2\cap E_2} & \le &  C_1 \|Y^{\le 2} (\nabla \phi,\psi)\|_{L^2_x}+C_1\|Z^{\le 2} \Box_g u\|_{L^1_t L^2_x} \\
 & & +C_1\|Z^{\le 1}\Box_g u(0,x)\|_{L^2_x}+\|\pa^{\le 1} \Box_g u\|_{(L^\infty_t\cap L^2_t )L^2_x(B_{2R})}
\end{eqnarray*}

We shall argue inductively to prove that there exists $\ep_0>0$ such that for any $\ep\le\ep_0$, we have
\beeq\label{eq-unif2}\|u_{k}\|_{LE_2\cap E_2} \le 3 C_1 \varepsilon, \eneq
for all $k\ge 1$. It has been true for $k=1$. 
For $k\ge 1$, assume we have \eqref{eq-unif2} for any $u_j$ with $j\le k$.

At first, by Proposition \ref{thm-1212}, we know that
\beeq\label{eq-bd-21}
\|Z^{\le 1} u\|_{L^\infty_t L^\infty_x}\les \|u\|_{E_2}\ ,\eneq
which particularly gives us $|Z^{\le 1} u_{j}(t,x)|\les \ep, j\le k$ by the induction assumption. 

It follows from the definition of $F_p$ that,
$$|Z^{\le 1} F_p(u_k, \pt u_k)(t,x)|\le C(\|u_k(t,\cdot)\|_{L^\infty_{x}})|\pa u_k|^{p-1}(|Z^{\le 1} u_k \pa u_k|+|Y^{\le 1}\pa u_k|+|\pt^2 u_{k}|)\ ,$$
where $C(t)$ is a continuous increasing function. Thus, if $\ep\le 1$, we have
$$\sum_{|\al|\le 1}\|Z^\al F_p(u_k, \pt u_k)(0,\cdot )\|_{L^2_x}\les
\ep^{p-1}(\ep+\|\pt^2 u_{k}(0,\cdot)\|_{L^2_x})\ .
$$
For $\|\pt^2 u_{k}(0,\cdot)\|_{L^2_x}$, using the definition of $u_k$, we get
$$\|\pt^2 u_{k}(0,\cdot)\|_{L^2_x}\les
\|\Delta_g u_{k}(0,\cdot)\|_{L^2_x}+
\|F(u_{k-1}, \pt u_{k-1})(0,\cdot)\|_{L^2_x}
\les \ep+\ep^p\les \ep\ ,
$$ and so
\beeq\label{0203-1}\sum_{|\al|\le 1}\|Z^\al F_p(u_k, \pt u_k)(0,\cdot )\|_{L^2_x}\les \ep^p\ .\eneq

Similarly, 
$|\pt^2 u_k|\les |\Delta_g u_k|+|\pa u_{k-1}|^p
\les |\nabla\pa u_k|+|\pa u_{k-1}|
$,
and so
$$|\pa^{\le 1} F_p(u_k, \pt u_k)(t,x)|\le \tilde C(\|u_k\|_{L^\infty_{t,x}})|\pa u_k|^{p-1}(|\nabla^{\le 1}\pa u_k|+|\pa u_{k-1}|)\ ,$$
for some continuous increasing function $\tilde C$.
Then $$\|\pa^{\le 1} F_p(u_k, \pt u_k)\|_{(L^2_t\cap L^\infty_t )L^2_x(B_{2R})}\les (\|u_k\|_{E_2}+\|u_{k-1}\|_{E_2})^{p-1}
\|u_k\|_{LE_2\cap E_2}\les \ep^p\ .
$$

Summarizing the above estimates, there exists $C_2$ such that
\beeq\label{eq-620-15}
\|u_{k+1}\|_{LE_2} \le C_1\varepsilon +C_2 \ep^p+ C_1  \|Z^{\le 2} F_p(u_k, \pt u_k)\|_{L^1_tL^2_x}.\eneq

By \eqref{eq-620-15}, to complete the proof \eqref{eq-unif2}, it suffices to show
\beeq\label{eq-unif}
\sum_{|\al|\le 2}\|Z^\al F_p(u, \pt u)\|_{L^1_tL^2_x}
\les \|u\|_{LE_2\cap E_2}^p\ .\eneq

Notice that there exist smooth functions $b_i$, $1\le i\le 5$, such that
\begin{eqnarray*}
& | Z^{\le 2} F_p(u, \pt u)| \le &|b_1(u)| |\pa u|^{p-1} |Z^{\le 2} \pa u|+|b_2(u)| |\pa u|^{p-2} |Z^{\le 1} \pa u|^2\\
&&+|b_3(u)| |\pa u|^{p-1} |Z u| |Z^{\le 1} \pa u|+|b_4 (u)| |\pa u|^p |Z u|^2+|b_5(u)| |\pa u|^p |Z^2 u|\ .
  \end{eqnarray*}
By Lemma \ref{thm-Sobo}, we have
\beeq\label{eq-bd} |\pa u|\les \frac{\|u\|_{E_{2}}}{\<r\>} ,\ |Z u|\les \|u\|_{E_{2}}\ .\eneq
Using \eqref{eq-bd-21}, smoothness of $b_i$ and \eqref{eq-bd}, we see that $u$ is bounded and
$$
|  Z^{\le 2} F_p(u, \pt u) |\les |\pa u|^{p-1} (|Z^{\le 2} \pa u|+|Z^{\le 2} u|/\<r\>)+
|\pa u|^{p-2} |Z^{\le 1} \pa u|^2\ .$$

The first term can be dealt with as follows, by \eqref{eq-bd}, Lemma \ref{thm-Sobo}, and the fact that $p>2$,
\begin{eqnarray*}
  &&\||\pa u|^{p-1} (|Z^{\le 2} \pa u|+|Z^{\le 2} u|/\<r\>)\|_{L^1_t L^2_x} \\
  &\les& \|\<r\>\pa u\|_{L_t^\infty L^\infty_x}^{p-2}\|\<r\>^{(3-p)/2}\pa u\|_{L^2_t L^\infty_x}\|  \<r\>^{-(p-1)/2}\left(|Z^{\le 2} \pa u|+\frac{|Z^{\le 2} u|}{\<r\>}\right)\|_{L^2_t L^2_x}\\
  &\les& \|u\|_{E_2}^{p-2}\|  \<r\>^{-1/2-(p-2)/2}(|Z^{\le 2} \pa u|+|Z^{\le 2} u|/\<r\>)\|_{L^2_t L^2_x}^2\\
  &\les & \|u\|_{LE_{2}}^{2}\|u\|_{E_2}^{p-2} \ .
\end{eqnarray*}  Similarly, for the second term, we get
\begin{eqnarray*}
\||\pa u|^{p-2} |Z^{\le 1} \pa u|^2\|_{L^1_t L^2_x}
  &\les&\|\<r\>\pa u\|_{L_t^\infty L^\infty_x}^{p-2} \|  \<r\>^{-(p-2)/2} Z^{\le 1} \pa u \|_{L^2_t L^4_x}^2\\
    &\les& \|u\|_{E_{2}}^{p-2}\|  \<r\>^{-(p-2)/2-1/2}Z^{\le 2} \pa u\|_{L^2_t L^2_{x}}^2\\
        &\les& \| u\|_{LE_{2}}^2\|u\|_{E_2}^{p-2}\ .
\end{eqnarray*}
This finishes the proof of  \eqref{eq-unif} and so is the uniform boundedness \eqref{eq-unif2}.

Similar proof will give us the convergence of the sequence $\{u_k\}$
$$\|u_{k+1}-u_k\|_{LE\cap E}\le C\|F_p(u_k)-F_p(u_{k-1})\|_{L^1_tL^2_x}\le
\frac{1}2\|u_k-u_{k-1}\|_{LE\cap E}$$
provided that $\ep_0$ is small enough.

Together with the uniform boundedness \eqref{eq-unif2}, we find an unique global solution $u\in L^\infty_t H^3\cap Lip_t H^2$ with $\|u\|_{LE_2\cap E_2}\le 3 C_1 \ep$.
Strictly speaking,
to complete the proof, we need also to prove the regularity of the solution $u\in C_t H^3\cap C^1_t H^2$. As it is standard, we omit details here, and refer the reader to the end of Section 4 in \cite{Wa13} or \cite{HWY11} P533.

For the remaining case,
$p=2$, 
we need only to notice that by Proposition \ref{thm-LE-h},  we have  for any $T\ge 2$
\begin{eqnarray*}
\|u\|_{LE_2\cap E_2}
&\les& \sum_{|\al|\le 2}\|(\nabla,\Omega)^\al (\nabla \phi,\psi)\|_{L^2_x}+\|Z^\al F
\|_{L^1_T L^2_x}
\\
&&
+ \sum_{|\ga|\le 1}\|Z^\ga F(0,x)\|_{L^2_x}+\| \pa^\ga F\|_{(L^\infty_T\cap L^2_T )L^2_x(B_{2R})}\nonumber
\end{eqnarray*}
for solutions to \eqref{wave} in $[0,T]\times M$. 
Previous proofs, together with Lemma \ref{KSS-dual} (1.9), give us
$$\|u\|_{LE_2\cap E_2}\les \ep
+\tilde C(\|u\|_{E_2})\|u\|_{LE_2\cap E_2}^2
+\|Z^{\le 2} \pa u\|_{l^{-1/2}_2(L^2_T L^2_x)}^2
\les \ep
+(\tilde C(\|u\|_{E_2})+\ln T)\|u\|_{LE_2\cap E_2}^2
$$
which essentially give the almost global existence, as long as $\ep^2 \ln T\ll \ep$, i.e., $T\le \exp (c/\ep)$ with small enough $c>0$.

\section{Radial Glassey conjecture} \label{sec-rad}
In this section, we give the proof for Theorem \ref{thm-sub}, based on
Lemma \ref{thm-trace}, Proposition \ref{thm-LE-h} and Lemma \ref{KSS-dual}.

We set $u_0\equiv 0$ and recursively define $u_{k+1}$
to be the solution to the linear equation
\beeq\label{3iterate}
  \Box_g u_{k+1} = G_p(u_{k}, \pt u_k), u_{k+1}|_{x\in \pa B_1}=0,    u_{k+1}(0,x)=\phi, 
  \pt u_{k+1}(0,x) = \psi.
\eneq
By assumption, $u_k$ are radial functions.

\subsection{Global existence}
Recall Lemma \ref{thm-trace}, $M=\{|x|>1\}$, where $r\sim \<r\>$ and the fact that $u$ is radial,  we have
\beeq\label{eq-trace-rad2}
\|\<r\>^{(n-1)/2} \pa u\|_{L^\infty_x(M)}\les \|u\|_{E_1}\ .
\eneq

By the smallness condition \eqref{b-smallness} on the data and the equation,  we know from the definition of $G_p$ that, 
 for $\ep$ small enough, we have
$$\|G_p(u_k, \pt u_k)(0,\cdot )\|_{L^2_x}\les
\ep^{p}\les \ep
$$
$$\|G_p(u_k, \pt u_k)\|_{(L^2_t\cap L^\infty_t )L^2_x(B_{2R})}\les
\|u_k\|_{E_1}^{p-1}\|u_k\|_{LE\cap E}\ .
$$

With the above estimates, it follows from Proposition \ref{thm-LE-h} that there is a constant $C_3$ so that $\|u_1\|_{LE_1} \le C_3\ep$, and
\beeq\label{eq-620-15-2}
\|u_{k+1}\|_{LE_1} \le C_3\varepsilon + C_3 \|\pa^{\le 1} G_p(u_k, \pt u_k)\|_{LE^*+L^1_t L^2_x}+C_3
\|u_k\|_{E_1}^{p-1}\|u_k\|_{LE\cap E}
.\eneq

As in Section \ref{sec-Gla-3}, for global existence, we need to 
prove the uniform boundedness and convergence of the iteration series $u_k$. Here, we only give the proof of the uniform boundedness, which could be reduced to the proof of
\beeq\label{eq-radG-1}\|G_p(u, \pt u)\|_{{LE}_1^*}\les \|u\|_{{LE}_1\cap E_1}^p\ ,
\eneq for any $p>p_c$ and radial $u\in LE_1\cap E_1$.
In fact, by Lemma \ref{thm-trace}, we have
\begin{eqnarray*}
\|G_p(u, \pt u)\|_{{LE}_1^*}&=&
\| \pa^{\le 1} G_p(u, \pt u)\|_{l^{1/2}_1 L^2 L^2}
\\
&\les&\|  |\pa u|^{p-1} \pa^{\le 1} \pa u\|_{l^{1/2}_1 L^2 L^2}
\\
&\les&
\|\<r\>^{(n-1)/2}\pa u\|_{L^\infty_{t,x}}^{p-1}
\| \<r\>^{-(n-1)(p-1)/2}  \pa^{\le 1} \pa u\|_{l^{1/2}_1L^2 L^2}\\
&\les&\|u\|_{E_1}^{p-1}
\| \<r\>^{-(n-1)(p-1)/2}  \pa^{\le 1} \pa u\|_{l^{1/2}_1L^2 L^2}
\\
&\les&
\|u\|_{E_1}^{p-1}\|\pa^{\le 1} \pa u\|_{l^{-1/2}_\infty L^2 L^2}
\les
\|u\|_{LE_1\cap E_1}^{p}
\end{eqnarray*}
provided that
$(n-1)(p-1)/2>1$, that is $p>p_c$.

\subsection{The critical case}\label{sec-rad2}
For the critical case $p= p_c$, by \eqref{eq-620-15-2} and Lemma \ref{KSS-dual} (1.11), we have,  for any $T\ge 2$,
$$\|u_{k+1}\|_{LE_1\cap E_1}
\le C_3\ep+C (\ln T)^{1/2} \|\pa^{\le 1} G_p(u_k,\pt u_k)\|_{l^{1/2}_2 L^2_T L^2_x }
+C_3\|u_k\|_{E_1}^{p-1}\|u_k\|_{LE\cap E}\ .
$$

Since $p=p_c$, i.e., $(n-1)(p-1)/2=1$, we have
\begin{eqnarray*}
\|\pa^{\le 1}G_p(u, \pt u)\|_{l^{1/2}_2 L^2_T L^2_x }
&\les&
\|  |\pa u|^{p-1} \pa^{\le 1} \pa u\|_{l^{1/2}_2 L^2 L^2}
\\
&\les&
\|\<r\>^{(n-1)/2} \pa u\|_{L^\infty_{t,x}}^{p-1}
\| \<r\>^{-(n-1)(p-1)/2}  \pa^{\le 1} \pa u\|_{l^{1/2}_2 L^2 L^2}\\
&\les&\|u\|_{E_1}^{p-1}
\|  \pa^{\le 1} \pa u\|_{l^{-1/2}_2L^2 L^2}
\\
&\les&
 (\ln T)^{1/2} \|u\|_{LE_1\cap E_1}^{p}\ ,
\end{eqnarray*}
where we have used Lemma \ref{KSS-dual} (1.9) in the last step.
In conclusion, we have obtained
$$\|u_{k+1}\|_{LE_1\cap E_1}
\le C_3\ep+C   \|u_k\|_{LE_1\cap E_1}^p \ln T\ ,
$$
which essentially give rise to the almost global existence, by choosing
$T$ such that $\ep^p \ln T\ll \ep$, that is,
$T=\exp(c \ep^{1-p})$ with certain small enough $c$.

\subsection{The case $p<p_c$}\label{sec-rad3}
Similarly, for $1<p<p_c$, we have $\mu=(n-1)(p-1)/4\in (0,1/2)$.
By \eqref{eq-620-15-2} and Lemma \ref{KSS-dual} (1.12), we get for any $T\ge 2$,
$$\|u_{k+1}\|_{LE_1\cap E_1}
\le C_3\ep+C T^{1/2-\mu} \|\pa^{\le 1} G_p(u_k,\pt u_k)\|_{l^{\mu}_2 L^2_T L^2_x}
+C_3\|u_k\|_{E_1}^{p-1}\|u_k\|_{LE\cap E}\ .
$$
As before, by Lemma \ref{KSS-dual} (1.10),
\begin{eqnarray*}
\|\pa^{\le 1} G_p(u, \pt u)\|_{l^{\mu}_2(L^2_T L^2_x)}&\les&
\|  |\pa u|^{p-1} \pa^{\le 1} \pa u\|_{l^{\mu}_2 L^2 L^2}
\\
&\les&
\|\<r\>^{(n-1)/2} \pa u\|_{L^\infty_x}^{p-1}
\| \<r\>^{-(n-1)(p-1)/2}  \pa^{\le 1} \pa u\|_{l^{\mu}_2 L^2 L^2}\\
&\les&\|u\|_{E_1}^{p-1}
\|  \pa^{\le 1} \pa u\|_{l^{-\mu}_2L^2 L^2}
\\
&\les&
T^{1/2-\mu} \|u\|_{LE_1\cap E_1}^p\ .
\end{eqnarray*}
With the above two estimates,
we get
$$\|u_{k+1}\|_{LE_1\cap E_1}
\le C_3\ep+C T^{1-2\mu} 
\|u_k\|_{LE_1\cap E_1}^p\ ,
$$
and then the long time existence in the interval $[0,T]$ could essentially be proved, by setting
$T$ such that
$\ep^p T^{1-2\mu}\ll \ep$, i.e.,
$$T=c \ep^{\frac{2(p-1)}{(n-1)(p-1)-2}}$$ with small enough $c$.

\medskip 
{\bf Acknowledgements.}
The author would like to thank the anonymous referee for valuable comments, which helped improve the manuscript.


\end{document}